\documentclass[12pt]{article}

\usepackage{amsmath,amsthm,amsfonts,amssymb, color,xcolor,subcaption,graphicx} 

\usepackage{varioref}

\newtheorem{theorem}{Theorem}[section]
\newtheorem{corollary}[theorem]{Corollary}
\newtheorem{example}[theorem]{Example}
\newtheorem{assumption}[theorem]{Assumption}

\newtheorem{proposition}[theorem]{Proposition}

\newtheorem{definition}[theorem]{Definition}

\def\cB{\mathcal{B}}

\def\cD{\mathcal{D}}

\def\cF{\mathcal{F}}

\def\cL{\mathcal{L}}

\def\cO{\mathcal{O}}

\def\bE{\mathbb{E}}

\def\bN{\mathbb{N}}
\def\bP{\mathbb{P}}
\def\bR{\mathbb{R}}

\usepackage{float}
\begin{document}

\title{ Pathwise regularity of solutions for a class of elliptic SPDEs with symmetric L\'evy noise}

\author{ Juan J. Jim\'enez \footnote{University of Ottawa, Department of Mathematics and Statistics, 150 Louis Pasteur Private, Ottawa, Ontario, K1G 0P8, Canada. E-mail address: jjime088@uottawa.ca.} }

\date{July 21, 2025}
\maketitle

\begin{abstract}
\noindent In this article, we investigate the existence and uniqueness of random‐field solutions to the elliptic SPDE $-\mathcal{L}u=\dot{\xi}$ on a bounded domain $D$ with Dirichlet boundary conditions $u=0$ on $\partial D$, driven by symmetric L\'evy noise $\dot{\xi}$. Under general sufficient conditions on the coefficients of the second‐order operator $\mathcal{L}$, we prove the existence of a mild solution via the corresponding Green's function and show that the same framework applies to the spectral fractional Laplacian of power $\gamma \in (0, \infty)$. In particular, whenever $\gamma>\tfrac d2$, the solution admits a continuous modification.

\end{abstract}


\vspace{1mm}


\section{Introduction}
Let \((\Omega, \cF, \bP)\) be a complete probability space. Let \(D\) be a bounded domain in \(\bR^d\), and denote by \(\partial D\) its boundary, assumed to be of class \(C^\infty\). Additionally, let \(\cO\) be the \(\sigma\)-field on \(\Omega \times D\) generated by \(\cF \times \cB(D)\), where \(\cB(D)\) is the Borel \(\sigma\)-field on \(D\). In this article, we study the pathwise regularity properties of the solution to the stochastic linear elliptic equation driven by a Lévy noise:
\begin{equation}
\label{SPDE1}
    \begin{cases}
        -\cL u(x) = \dot{\xi}(x), \,&  x \in D,\\
        u(x) = 0, \,&  x \in\partial D,
    \end{cases}
\end{equation}
where \(\cL\) is a \emph{second-order linear elliptic operator}, and \(\dot{\xi}\) denotes a \emph{symmetric white Lévy noise} on \(D\) in the sense of \cite{RR89}, i.e.,  \(\dot{\xi} = \{\xi(A)\}_{A\in\cB(D)}\) is an \emph{independently scattered random measure} satisfying
\begin{equation}
\label{levy-eq}
\mathbb{E}\bigl[e^{i\,u\,\xi(A)}\bigr] \;=\; \exp\bigl(|A|\;\Psi(u)\bigr),
\quad \text{for all } A\in\cB(D),
\end{equation}
where \(\cB(D)\) is the Borel \(\sigma\)-algebra on \(D\), \(\lvert A\rvert\) denotes Lebesgue measure of \(A\), and
\[
\Psi(u)
= i\,b\,u \;-\;\frac{\sigma^2\,u^2}{2}
\;+\;\int_{\mathbb{R}}\Bigl(e^{i\,u\,z}-1 - i\,u\,z\,\mathbf{1}_{\{\lvert z\rvert\le1\}}\Bigr)\,\nu(dz),
\]
for some $b \in \bR$ and $\sigma>0$. Here, \(\nu\) is a Lévy measure on \(\mathbb{R}_0 := \mathbb{R}\setminus\{0\}\), satisfying
\[
\int_{\mathbb{R}_0}\bigl(|z|^2\wedge1\bigr)\,\nu(dz)<\infty,
\]
and symmetric in the sense that \(\nu(A)=\nu(-A)\) for all Borel sets \(A\subset\mathbb{R}_0\). We say that \(\dot{\xi}\) has  \textit{characteristic triplet} \((b,\sigma,\nu)\) if it satisfies \eqref{levy-eq}. Observe that $\dot{\xi}$ satisfies {\em L\'evy It\^o decomposition}: for all $A \in \cB (D)$,
\begin{equation}
    \label{levy-ito}
    \xi(A) = b |A| + \sigma W_\xi(A) + \int_{D \times \{|z| \le 1 \} } 1_A(x) z \widetilde J_{\xi}(dx,dz) + \int_{D \times \{|z| > 1 \} } 1_A (x) z J_{\xi}(dx,dz),
\end{equation}
where \(J_{\xi}\) is the Poisson random measure associated to \(\xi\) with intensity \(\mu(dy,dz)=dy\,\nu(dz)\), and \(\widetilde J_{\xi}=J_{\xi}-\mu\) its compensator. Finally, \(W_{\xi}\) denotes the continuous Gaussian white noise component of \(\xi\), with covariance,
\[
\mathbb{E}\bigl[W_{\xi}(A)\,W_{\xi}(B)\bigr]
= \sigma^2 \,\lvert A\cap B\rvert \, \, \, \, \text{for all $A,B \in \cB(D).$}
\]

We make the following clarifications. If $d=1$, we let $D = (a,b)$ for some $a, b \in \bR$ with $a < b$. We assume that $\cL$ is a second-order linear elliptic operator of the form
\begin{equation}
    \label{linear-elliptic}
    \cL = \sum_{i, j =1}^d a_{i,j}(x) \partial_{x_i x_j}  + \sum_{i=1}^d \tilde{b}_i (x) \partial_{x_i},
\end{equation}
satisfying the \textit{uniform ellipticity condition}:
\begin{equation}
    \label{u-eli}
    \kappa \| \zeta \|_{\bR^d}^2 \le \sum_{i, j =1 }^d a_{i,j}(x) \zeta_i \zeta_j \le \Lambda \| \zeta \|_{\bR^d}^2, \quad \text{for all } \zeta = (\zeta_1,\ldots,\zeta_d) \in \bR^d,\ x \in D,
\end{equation}
for some constants $0 < \kappa < \Lambda$, where \(\|\cdot\|_{{\bR^d}}\) is the Euclidean norm on \(\mathbb{R}^d\). We also assume that the entries of the matrix \(\mathbf{A}(x) = \{a_{i,j}(x)\}_{i,j}\) belong to \(L^{\infty}(D)\). Additionally, we assume that \(\tilde{b}_i \in L^{\infty}(D)\) and \(\tilde{b}_i \ge 0\) for all \(i=1,\dots,d\).

 In the special case \(\mathbf{A}(x) = I_{d\times d}\), the identity matrix, the operator \(\cL\) reduces to the Laplace operator,
\[
\Delta = \sum_{i=1}^d \frac{\partial^2}{\partial x_i^2}.
\]
We denote by \(G_{\Delta}\) the Green function of \(-\Delta\).

In this work we use the following notation:
\begin{itemize}
    \item \(C^0(D)\) is the space of continuous functions on \(D\).
    \item \(L^p(D)\) is the usual Lebesgue space on \(D\).
    \item We denote by \(\mathrm{div}\) the divergence operator and by \(\nabla\) the gradient. Their norms in \(L^p(D)\) are \(\|\cdot\|_p\), and in \(L^\infty(D)\) it is \(\|\cdot\|_\infty\).
    \item \(\overline{D}\) is the topological closure of \(D\) in the usual topology.
    \item Here, \(\|\cdot\|\) denotes the usual Euclidean matrix norm on \(\mathbb{R}^{d\times d}\), defined by \(\|A\| = \sup_{ \|x\|_{\bR^d} =1} \|Ax \|_{{\bR^d}}\).

\end{itemize}

It is well known that SPDEs can be formulated and analyzed via different frameworks. One is the \textit{random field approach} originating in the seminal work \cite{walsh86} by Walsh, and another is the \textit{semigroup approach} by Da Prato and Zabczyk (see \cite{DZ14}). Both approaches have been successfully developed in parallel; however, there are few investigations that study the connections between them. One such work is \cite{dalang-quer11}. In particular, the setting of this work is the random field approach together with the stochastic integration techniques of \cite{RR89}. We say that a $\cO$-measurable stochastic process $u= \{ u(x) \; ;  \; x \in D \}$ is a \textit{mild solution} if it satisfies:
\begin{equation}
    \label{mild-spde}
    u(x) = \int_D G_{\cL} (x,y) \xi(dy),
\end{equation}
where $G_{\cL} (x, \cdot)$ is the Green's function of $-\cL$ embedded in Dirichlet conditions. The goal of this article is to show the existence of a mild solution of \eqref{SPDE1} that satisfies:
\begin{itemize}
    \item[(i)] \(u\) is a \emph{generalized solution} of \eqref{SPDE1}, in the sense of Definition 3.4 in \cite{DH}. 
    \item[(ii)] The map \(\omega\in\Omega\mapsto u(\omega,\cdot)\) takes almost surely values in a \emph{fractional Sobolev space} of order \(r\in\mathbb{R}\).
    \item[(iii)] We extend the results for elliptic operators to a class of spectral operators. In particular, we work with a spectral power of the Laplace operator.
\end{itemize}

For our main results, we only require an integrability condition on the Lévy measure \(\nu\) for the small jumps \(\{|z|\le1\}\), as shown in the following table. \hfill \break

\begin{table}[h!]
    \centering
    \setlength{\tabcolsep}{9pt} 

\begin{tabular}{|c|c|c|c|}
  \hline
  \textbf{Dimension \( d \)} 
    & \textbf{Triplet \( (b,\sigma,\nu) \)} 
    & \textbf{Lévy Measure Condition} 
    & \textbf{Value of \(p\)} \\ 
  \hline
  \( d \le 3 \) 
    & \((b,\sigma,\nu)\) 
    & \(\displaystyle \int_{\{|z|\le1\}}|z|^p\,\nu(dz)<\infty\) 
    & \(p=2\) \\ 
  \hline
  \( d \ge 4 \) 
    & \((b,0,\nu)\) 
    & \(\displaystyle \int_{\{|z|\le1\}}|z|^p\,\nu(dz)<\infty\) 
    & \(0<p<\tfrac{d}{d-2}\) \\ 
  \hline
\end{tabular}
\caption{Existence conditions for mild solutions to \eqref{SPDE1} as given in Theorem \ref{existence-1}, assuming the singularity of \(G_{\cL}(x,y)\) at \(x=y\) behaves as that of \(G_{\Delta}(x,y)\).}
\end{table}

Although there is no direct reference for the path properties of solutions to \eqref{SPDE1}, the study of continuity for spatial heavy-tailed random fields has been active since the late 20th century. In \cite{KM91}, Hölder continuity of sample paths was proved for a class of one-dimensional self-similar stable processes. Similarly, in the one-dimensional setting, \cite{CC92} established the modulus of continuity for non-Gaussian \(L^2\)-random fields. In \cite{MR05}, continuity and boundedness of spatial infinitely divisible fields were proved for stochastic Wiener integrals with respect to a Poisson random measure. Finally, \cite{xiao2010} proved the uniform modulus of continuity for heavy-tailed random fields in \(\mathbb{R}^d\), including \(\alpha\)-stable random fields.

Finally, in the remainder of this section,  we briefly discuss the regularity theory for deterministic elliptic PDEs and provide the main references relevant to \eqref{SPDE1}.

Stochastic elliptic PDEs on bounded domains have been widely studied since Walsh’s seminal work \cite{walsh86}. For instance, the existence and regularity of solutions to the quasilinear elliptic problem, for \(d \le 3\),
\begin{equation}
\label{poisson-1}
    \begin{cases}
        -\Delta u (x) = b(u(x)) + \sigma \dot{W}(x), \,& x \in D,\\
        u(x) = 0, \,& x \in \partial D,
    \end{cases}
\end{equation}
have been investigated in \cite{BGN80,BP90,DN92,R80,NT95,T98} in the case where \(W\) is \emph{white Gaussian noise}, \(\sigma \in \mathbb{R}\), and \(b\colon \mathbb{R} \to \mathbb{R}\). In particular, \cite{BP90} proves the existence and uniqueness of a solution that is almost surely continuous on \(\overline{D}\). In \cite{walsh86}, it was shown that, when \(b=0\), the solution of \eqref{poisson-1} is a generalized solution and takes values in a suitable Sobolev space.

More recently, in \cite{sole-viles18} it was shown that a solution of systems of the form \eqref{poisson-1} possess Hölder‐continuous paths, and upper and lower bounds were established for the corresponding hitting probabilities.

In contrast to the extensive study of \eqref{SPDE1} and \eqref{poisson-1} when the noise \(\dot{\xi}\) is a Gaussian white noise, there are few works that systematically study \eqref{SPDE1} in the case where \(\dot{\xi}\) has a nontrivial jump component. To the best of our knowledge, \cite{lokka} is the only work in literature that have as a main focus of study of equation \eqref{SPDE1} where $\dot{\xi}$ is a pure-jump L\'evy noise, under a \textit{finite-variance setting}, i.e.,
\begin{equation}
    \label{finite-var}
    \int_{\bR_0} |z|^2 \nu (dz) < \infty \Leftrightarrow \bE[|\xi(A) |^2] < \infty, \; \; \text{for all $A \in \cB(D)$}.
\end{equation}
Here, $\xi(A) = \int_{A \times \mathbb{R}_0} z \,\hat{N}(dy,dz)$, where \(\hat{N}\) is the \textit{compensated random measure} associated with the \textit{Poisson random measure} (PRM) \(N\) on \(D \times \mathbb{R}_0\) having intensity \(\mu(dy,dz)=dy\,\nu(dz)\). In the same context, see \cite{LD18}.

Although, there are interesting cases of L\'evy noises satisfying \eqref{finite-var} such as the \textit{variance-gamma L\'evy noises}, many other case of interested does not satisfy \eqref{finite-var} such as the \textit{symmetry $\alpha$-stable L\'evy noise}, i.e., $b=\sigma=0$, and
\begin{equation}
    \label{stable}
   \nu(dz)=\nu_{\alpha}(dz)=\frac{1}{2}\alpha |z|^{-\alpha-1}dz, \quad \text{for $\alpha \in (0,2).$} 
\end{equation}
In this work, we construct a solution of \eqref{SPDE1} driven by a general symmetric L\'evy noise $\dot{\xi}$  that covers the cases \eqref{finite-var} and \eqref{stable}. Moreover, under this framework, we do not require that $\dot{\xi}$ has moments of any order.

Regarding the analogue of \eqref{SPDE1} on the whole space \(\mathbb{R}^d\), i.e.
\begin{equation}
\label{spde-Rd}
- \cL u(x) = \dot{\xi}(x), \quad x \in \mathbb{R}^d,
\end{equation}
has been studied in \cite{berger,DH}. Interestingly, Theorem 6.13 in \cite{DH} establishes that no mild solution to \eqref{spde-Rd} exists when \(\dot{\xi}\) is a symmetric \(\alpha\)-stable Lévy noise on \(\mathbb{R}^d\) for any \(d \ge 1\). The existence of a generalized solution of \eqref{spde-Rd} and of second‐order elliptic SPDEs on \(\bR^d\) has been studied in \cite{berger}, along with moment estimates. Furthermore, \cite{berger} also established the existence of a mild solution in \(d=3\) for the electric Schr\"odinger equation under suitable moment conditions over the large jumps of the L\'evy measure.

 On  the other hand, in \cite{CAD22} the solutions of SPDEs driven by a spatial \(L^2\)-random field with specified functional covariance were studied, including \eqref{spde-Rd}, as well as the heat and wave equations.

The path properties of parabolic SPDEs driven by a general Lévy white noise in space–time were studied in \cite{CDH19}.

\section{Preliminaries}

In this section, we briefly discuss the stochastic integration framework used in this work, and the basic ingredients of fractional Sobolev spaces in bounded domains.

\subsection{Stochastic integration}

Although various authors had introduced stochastic integrals for random measures since Paul Lévy’s seminal work \cite{levy37}, the theory in \cite{RR89} was pioneering in treating integrators that may not depend on time and instead rely solely on the topology of the spatial component, making this approach particularly suitable for solving elliptic SPDEs.

We say that $f$ is a {\em simple function} if it is given by
\begin{equation}
\label{NR-simple}
f(x) = \sum_{i=1}^n \alpha_i 1_{A_i}(x),
\end{equation}
where $A_1,\ldots, A_n$ are disjoint sets in $\mathcal{B}(D)$ and $\alpha_1,\ldots,\alpha_n \in \bR$. We denote by $\overline{\mathcal{S}}$ the collection of simple functions of the form \eqref{NR-simple}. If $f$ is a simple function of the form
\eqref{NR-simple}, we define the stochastic integral of $f$ with respect to $\dot{\xi}$ by:
\[
\langle \dot{\xi},f \rangle = \int_D f(x) \xi(dx) =  \sum_{i=1}^n \alpha_i \xi(A_i),
\]
and stochastic integral of $f$ on the set $A\in \mathcal{B}_b(D)$, with respect to $\dot{\xi}$, by:
\[
\langle \dot{\xi},f 1_A \rangle =  \int_A f(x) \xi(dx) =  \sum_{i=1}^n \alpha_i \xi(A\cap A_i).
\]
Note that $\langle \dot{\xi},1_A \rangle = \xi(A)$ for all $A \in \cB( D )$.
\begin{definition}
    A Borel measurable function $f: D \to \bR$ is  \textit{$\dot{\xi}$-integrable} if there exists a sequence $\{ f_n \}_{n \ge 0}$ in $\overline{\mathcal{S}}$ such that
    \begin{itemize}
        \item[i)] $f_n \to f$ a.s. for $n \to + \infty$,
        \item[ii)] for every $A \in \mathcal{B} (D)$, the sequence $\{ \langle \dot{\xi},f_n 1_A \rangle \}_{n \ge 0}$ converges in probability, as $n \to + \infty$.
    \end{itemize}
If $f$ is $\dot{\xi}$-integrable, we set
\begin{equation}
\label{stochastic-int}
\langle \dot{\xi},f \rangle  \stackrel{\bP}{=} \lim_{n \to + \infty} \langle \dot{\xi},f_n \rangle.
\end{equation}
\end{definition}

In \cite{RR89}, one finds a complete characterization of the deterministic functions integrable with respect to \(\dot{\xi}\). More precisely, by Theorem 2.7 in \cite{RR89}, a $\cB(D)$-measurable function \(f\colon D\to\mathbb{R}\) is \(\dot{\xi}\)-integrable if and only if
\begin{equation}
   \label{cond-i}
\int_D \bigl|b\,f(x)\bigr|\,dx < \infty,\quad
\int_D \bigl|\sigma\,f(x)\bigr|^2\,dx < \infty,
\quad\text{and}\quad
\int_{D\times\mathbb{R}_0}\bigl(|z\,f(x)|^2 \wedge 1\bigr)\,dx\,\nu(dz) < \infty.
\end{equation}
    Moreover, if $f$ is $\dot{\xi}$-integrable, then
    \begin{equation}
    \label{LK}
      \bE \left[ e^{i u \langle \dot{\xi},f \rangle } \right]=  \exp{\Big( \int_D \Psi ( f(x) ) dx  \Big)}.
    \end{equation}

\subsection{Fractional Sobolev spaces}

In this section, we briefly recall the main ingredients of fractional Sobolev spaces on \(D\). Consider the eigenvalue problem for the Dirichlet Laplacian:
\begin{align*}
    -\Delta e_n(x) & = \lambda_n e_n(x), \quad x \in D, \\
    e_n(x) & = 0, \quad \quad \quad \quad x \in \partial D,
\end{align*}
where \( \Delta \) is the Laplacian operator. The eigenvalues \( \{\lambda_n\}_{n=1}^\infty \) are positive, satisfy \( \lambda_n \to \infty \) as \( n \to \infty \), and the corresponding eigenfunctions \( \{e_n\}_{n=1}^\infty \) can be chosen to form an orthonormal basis of \( L^2(D) \) with respect to the inner product:
\[
    \langle f, g \rangle_{L^2(D)} = \int_D f(x) \overline{g(x)} \, dx.
\]
Let $g \in L^2 (D)$, we define the Fourier coefficients of $g$ with respect to \( \{e_n\}_{n=1}^\infty \) by
\[
\cF_k[g] =   \langle g, e_k \rangle_{L^2(D)}, \; \; \text{for each $k \in \bN$.}
\]

Let \(\mathcal{D}(D)\) denote the space of test functions on \(D\), and let \(\mathcal{D}'(D)\) be its dual, i.e., the space of distributions on \(D\). Abusing notation, we will use the same symbol as in \eqref{stochastic-int} for the distributional action on \(\cD(D)\).  That is, for any \(F\in\cD'(D)\), we write
\[
F(\phi)=\langle F,\phi\rangle,\quad\text{for all }\phi\in\cD(D).
\]

Let $E_0$ be the set of functions $f: D \to \mathbb{R}$ of the form
\[
f(x) = \sum^K_{i =1 } \alpha_k e_k (x), \quad \text{for $\alpha_i \in  \mathbb{R}$,}
\]
and define, for $r \in \mathbb{R}$,
\[
\| f \|_{H_r(D)} := \Bigg( \sum_{k = 1}^{\infty} \lambda_k^r | \cF_k [ f] |^2 \Bigg)^{1/2}.
\]

The completion of $E_0$ with respect to $\| \cdot \|_{H_r}$ is called the \textit{Sobolev space of order r}. We denote this space as $H_r (D)$. Each element $\Phi \in H_r(D)$ can be identified with a series of the form $ \Phi = \sum_{k = 1}^{\infty} a_k (\Phi) e_k$, where $a_k ( \Phi) \in \mathbb{R}$, and
\begin{equation}
    \label{sobolev-1}
    \| \Phi \|_{H_r(D)} : = \Bigg( \sum_{k=1}^{\infty} \lambda_k^r a_k ( \Phi )^2  \Bigg)^{1/2} < \infty. 
\end{equation}

The series defining $\Phi$ converges in the topology of  \(\mathcal{D}'(D)\) and in the $H_r(D)$-norm. Hence, 
\[
\mathcal{F}_j [\Phi]= \sum_{k =1}^{\infty} a_k (\Phi) \cF_j [e_k ] = a_j (\Phi).
\]
 By Parseval's identity, we can set $H_{-r}(D)$ as the dual of $H_r (D).$

The next result follows directly from Theorem 9.8 in Chapter 1 of \cite{lion-mag68} together with Lemma 2.18 in \cite{CDH19}.

\begin{theorem}
\label{cont-embdd}
If \(r > \tfrac{d}{2}\), then \(H_r(D)\subset C^0(\overline{D})\).
\end{theorem}

\begin{proof}
Let \(H^r(D)\) be the interpolation Sobolev space defined in (9.1) of Chapter 1 in \cite{lion-mag68}.  By Lemma 2.18 in \cite{CDH19}, one has \(H_r(D)\hookrightarrow H^r(D)\) for all \(r\ge0\).  The result then follows from Theorem 9.8 of Chapter 1 in \cite{lion-mag68}.
\end{proof}

\section{Stochastic elliptic PDEs}
In this section, we prove the existence of a random‐field solution and present examples of elliptic operators to which our theory applies.

\subsection{Existence of a solution}

In this section, we prove the existence of a mild solution to \eqref{SPDE1}, i.e.
\[
u(x) \;=\; \langle \dot{\xi},\,G_{\cL}(x,\cdot)\rangle.
\]
We begin by stating the main assumptions on \(\cL\), and we include a brief discussion of the basic ideas used throughout this section.

\begin{assumption} 
\label{ass-g}
    We assume that the Green function of $\cL$ satisfies the following properties:
    \begin{itemize}
        \item[(i)] For $d \ge 3$,
        \begin{equation}
            \label{d3g}
                G_{\cL} (x,y) \le C \|x - y \|_{\bR^d}^{2-d}, \quad \text{for all $x,y \in D$ with $x \neq y$.}
        \end{equation}
        \item[(ii)] For $d \leq 2$,
        \begin{equation}
            \label{d2g}
            G_{\cL} (x,y) \le C G_{\Delta} (x,y)  \quad \text{for all $x,y \in D$ with $x \neq y$}.
        \end{equation}
    \end{itemize}
\end{assumption}

If \(\dot{\xi}\) is a symmetric \(\alpha\)-stable Lévy noise, then under Assumption~\ref{ass-g} we have that
\[
\int_D G_{\cL}^\alpha(x,y) \,dy < \infty, 
\]
for all $x \in D$ for $\alpha$ satisfying the restriction on Table \ref{cond-alpha}. 

\begin{table}[h!]
        \centering
        \renewcommand{\arraystretch}{1.5} 
        \setlength{\tabcolsep}{15pt} 
        \begin{tabular}{|c|c|}
            \hline
            \textbf{Dimension \( d \)} & \textbf{Stability Parameter \( \alpha \)} \\ \hline
            \( d \leq 3 \) & \( 0 < \alpha < 2 \) \\ \hline
            \( d \geq 4 \) & \( 0 < \alpha < \frac{d}{d-2} \) \\ \hline
        \end{tabular}
        \caption{ \label{cond-alpha} Conditions for the existence of solutions to \eqref{SPDE1} when $\dot{\xi}$ is a is a symmetric \(\alpha\)-stable Lévy noise.}
\end{table}

It makes sense to use the notation \(\langle \cdot,\cdot\rangle\) for both the stochastic integral and the distributional action on \(\cD(D)\).  Indeed, consider the map
\[
F_{\dot{\xi}}\colon \Omega \times \cD(D) \;\longrightarrow\; \mathbb{R},
\qquad
F_{\dot{\xi}}(\omega,\varphi) = \big\langle \dot{\xi},\,\varphi\big\rangle (\omega).
\]
This map is well defined for all \(\varphi\in\cD(D)\) and is a random linear functional on \(\cD(D)\).  Moreover, \(F_{\dot{\xi}}\) is continuous in probability.  Hence, by Corollary 4.2 in \cite{walsh86},
\[
\mathbb{P}\bigl(F_{\dot{\xi}}\in \cD'(D)\bigr) = 1.
\]
In summary, \(\langle \dot{\xi},\cdot\rangle\) on \(\cD(D)\) is almost surely a distribution in \(\cD'(D)\).

On the other hand, for any $\varphi \in \cD(D)$, we have that
\begin{equation}
    \label{elp1}
    \begin{cases}
        - \mathcal{L} u (x)  = \varphi(x), \,  & x \in D,  \\
        u(x) = 0, \, & x \in \partial D,
    \end{cases}
\end{equation}
has a unique solution in $C^{\infty} ( \overline{D})$, and satisfies $u = G_{\cL} \circledast \varphi$, where
\[
( G_{\cL} \circledast \varphi) (x) = \int_D G_{\cL}(x,y) \varphi (y) dy.
\]

We say that a random linear functional \(u_{\mathrm{gen}}\) is a generalized solution to \eqref{SPDE1} if, for every \(\varphi \in \cD(D)\),
\begin{equation}
    \label{weak-s}
    \big\langle u_{\mathrm{gen}}, \varphi \big\rangle
    = \big\langle \dot{\xi},\,G_{\cL} \circledast \varphi \big\rangle.
\end{equation}

The question of when a mild solution is also a generalized solution has been studied in \cite{DH} for a wide variety of SPDEs, including \eqref{spde-Rd}. However, \cite{DH} does not cover elliptic SPDEs on bounded domains. One of the motivations of this section is to complement the results in \cite{DH} for \eqref{SPDE1}. Note that a generalized solution is also referred to as a weak solution.

\begin{theorem}
\label{existence-1}
Let \(\dot{\xi}\) be a symmetric Lévy white noise with characteristic triplet \((b,\sigma,\nu)\), and suppose Assumption \ref{ass-g} holds. Additionally, if \(d \ge 4\), assume \(\sigma = 0\) and
\begin{equation}
\label{con-int}
\int_{|z|\le 1} |z|^p \,\nu(dz) < \infty
\quad\text{for some }p \in \Bigl(0,\tfrac{d}{d-2}\Bigr).
\end{equation}
Then \eqref{SPDE1} admits a unique mild solution \(u\). Moreover, \(u\) is also a generalized solution to \eqref{SPDE1} in the sense that it satisfies \eqref{weak-s}.
\end{theorem}

\begin{proof}
   For \(d=1\), \(G_{\Delta}\) is bounded on \(D \times D\). Hence, condition \eqref{cond-i} is clearly satisfied. For $d \ge 2$: Set
    \begin{equation}
        \label{G1}
       A = \{ y \in D \setminus \{x \} : G_{\cL}(x,y) >1 \},
    \end{equation}
    and $A^c$ is the complement of $A$ with respect to $D \setminus \{x \}$.
    Note that
    \begin{equation}
    \label{iin1}
        \begin{split}
            \int_{D \times \bR } (|z G_{\cL}(x,y) |^2 \wedge 1 ) dy \nu (dz) & = \int_{ A^c \times \bR } (|z G_{\cL}(x,y) |^2 \wedge 1 ) dy \nu (dz) \\
            & + \int_{ A \times \bR } (|z G_{\cL}(x,y) |^2 \wedge 1 ) dy \nu (dz).
        \end{split}
     \end{equation}
    Note that the first integral on the right hand side of \eqref{iin1} is clearly finite, by the monotonicity of $x \in \bR_+ \to (1 \wedge |x|^2 )$.  On the other hand, note that  then
    \[
       \int_{  A \times \bR } (|z G_{\cL}(x,y) |^2 \wedge 1 ) dy \nu (dz) \le \int_{ A \times \bR } (|z G_{\cL}(x,y) |^p \wedge 1 ) dy \nu (dz).
    \]
    Moreover, since $( (c |z|^p )\wedge 1) \le c (|z|^p \wedge 1)$ for all $c \ge 1$, and $|G_{\cL}(x,y)|^p > 1$ for all $y \in A$,  we get that:
    \[
      \int_{ A \times \bR } (|z G_{\cL}(x,y) |^p \wedge 1 ) dy \nu (dz) \le  \int_{  A \times \bR } |G_{\cL}(x,y) |^p (|z |^p \wedge 1 ) dy \nu (dz).
    \]
    Note that 
    \[
     \int_{  A \times \bR } |G_{\cL}(x,y) |^p (|z |^p \wedge 1 ) dy \nu (dz) < \infty
    \]
    as long as there exists $p \in (0,2]$ such that
    \begin{equation}
      \label{cond-i-G}
      \int_{D }  |G_{\cL}(x,y) |^p dy < \infty \quad \text{and} \quad \int_{|z| \le 1} |z|^p \nu(dz) < \infty.
    \end{equation}
    These conditions are satisfied for $d =2,3$ if $p =2$, and for $d \ge 4$, we require
    \[
       0 < p < \frac{d}{d-2}.
    \]

    Now, we will show that $u = u_{\text{gen}}$. Observe that $G_{\cL} (x, \cdot)$ is $\dot{\xi}$-integrable, it satisfies the {\em Poisson $\dot{\xi}$-integrable}, i.e.
    \begin{equation}
        \label{poisson}
        \begin{aligned}
        \langle \dot{\xi},G_{\cL}(x,\cdot)\rangle
        &= \int_D b\,G_{\cL}(x,y)\,dy
        + \int_D \sigma\,G_{\cL}(x,y)\,W_\xi(dy)
        + \int_{D_s (x)} G_{\cL}(x,y)\,z\,\widetilde J_{\xi}(dy,dz)\\
        &\quad{}+ \int_{D_l (x)} G_{\cL}(x,y)\,z\,J_{\xi}(dy,dz)
        =: I_b(x)+I_W(x)+I_{s}(x)+I_{l}(x),
        \end{aligned}
    \end{equation}
  where $$D_s (x) = \{ (y,z) \in D \times \bR_0 ; \, |G_{\cL}(x,y) z | \le 1 \}$$ and $$D_l (x) = \{ (y,z) \in D \times \bR_0 ; \, |G_{\cL}(x,y) z | > 1 \}.$$
  
Now we show that for any \(\phi\in\mathcal{D}(D)\),
\[
\int_D \bigl\lvert \langle \dot{\xi},\,G_{\cL}(x,\cdot)\rangle\bigr\rvert\;\mu_{\phi}(dx)
<\infty
\quad\text{almost surely}.
\]
By Hölder’s inequality, this will hold as soon as 
\[
u(\omega,\cdot)\;=\;\langle \dot{\xi},\,G_{\cL}(\cdot,\cdot)\rangle
\;\in\;L^1(D)
\quad\text{for almost every }\omega\in\Omega.
\]
By Theorem 8.30 in \cite{GT2015}, {\em the torsion problem}
\begin{equation}
\label{torsion}
\begin{cases}
 -\cL v(x) = 1, & x\in D,\\
 v(x) = 0,       & x\in \partial D,
\end{cases}
\end{equation}
admits a unique generalized solution \(v\in C^0(\overline{D})\). Hence, the drift term 
$I_b(x) = b\,v(x)$ is well defined and continuous on \(D\). 

For the term \(I_W(x)\) when \(d\le3\), By It\^o’s isometry and Assumption \ref{ass-g}, 
\begin{equation}
\label{W-p}
\mathbb{E}\bigl[|I_W(x)|^2\bigr]
=\int_D G^2(x,y)\,dy
\le C_{D,d},
\end{equation}
where \(C_{D,d}>0\) is a constant independent of \(x\), which implies that $I_W (\omega, \cdot) \in L^1( D)$ for almost every $\omega \in \Omega$.

For $I_s (x)$, notice that
\[
\begin{split}
    I_s (x)  & = \int_{D_s (x)} G_{\cL}(x,y)\,z 1_{ \{ |z| \leq 1 \} }\,\widetilde J_{\xi}(dy,dz) \\
    & + \int_{D_s (x)} G_{\cL}(x,y)\,z 1_{ \{ |z| > 1 \} }\,\widetilde J_{\xi}(dy,dz). 
\end{split}
\]
By the symmetry of $\dot{\xi}$, and
\[
 \int_{D_s (x)} G_{\cL}(x,y)\, |z | 1_{ \{ |z| > 1 \} }\,  dy \nu (dz) \leq \mu \{(y,z) \in D \times \{|z| >1\} \},
\]
it follows that
\[
 \int_{D_s (x)} G_{\cL}(x,y)\,z 1_{ \{ |z| > 1 \} }\,\widetilde J_{\xi}(dy,dz) =  \int_{D_s (x)} G_{\cL}(x,y)\,z 1_{ \{ |z| > 1 \} }\, J_{\xi}(dy,dz).
\]
Then, by linearity of the Lebesgue integral,
\[
\begin{split}
    \int_D  \int_{D_s (x)} G_{\cL}(x,y)\, |z| 1_{ \{ |z| > 1 \} }\, J_{\xi}(dy,dz)dx & =\int_{D \times \bR_0} \Big( \int_D 1_{D_s (x)}(y,z) G_{\cL}(x,y) dx \Big)\, |z| 1_{ \{ |z| > 1 \} }\, J_{\xi}(dy,dz) \\
   & \le  \int_{D \times \bR_0} v (y)\, |z| 1_{ \{ |z| > 1 \} }\, J_{\xi}(dy,dz) < \infty \, \, \mbox{a.s.}
\end{split}
\]
Now, by It\^o isometry and \eqref{con-int}, we obtain that
\[
\begin{split}
    \bE \Bigg[ \Bigg| \int_{D_s (x)} G_{\cL}(x,y)\,z 1_{ \{ |z| \leq 1 \} }\,\widetilde J_{\xi}(dy,dz) \Bigg|^2 \Bigg] & \leq \int_{D \times \bR_0} |G_{\cL}(x,y)|^p |z|^p 1_{\{|z| \le 1 \} } dy \nu (dz) \\
    & \le C_{D,d,p} \int_{\{|z| \le 1\}} |z|^p \nu (dz),
\end{split}
\]
where $C_{D,d,p}$ is a constant independent of $x$. Hence, $I_s(\omega, \cdot) \in L^1 (D)$. For $I_l(x)$, note that
\[
\begin{split}
    \bE \left[ \Bigg| \int_{D_l (x)} G_{\cL}(x,y)\, z 1_{\{ |z| \le 1 \}} \,J_{\xi}(dy,dz) \Bigg|^p \right] & \leq\int_{D_l (x)} G_{\cL}^p (x,y)\, |z|^p 1_{\{ |z| \le 1 \}} dy \nu (dz) \\
    & \leq  C_{D,d,p} \int_{\{|z| \le 1\}} |z|^p \nu (dz).
\end{split}
\]
By the linearity of the Lebesgue integral, we get that
\[
\begin{split}
    \int_D \int_{D_l (x)} G_{\cL}(x,y)\, |z| 1_{\{ |z| > 1 \}} \,J_{\xi}(dy,dz) dx & \int_{D \times \bR_0} \Big( \int_D 1_{D_l (x)} (y,z) G_{\cL}(x,y) dx \Big) |z| 1_{\{ |z| > 1 \}} \,J_{\xi}(dy,dz)  \\
    & \le \int_{ D \times \bR_0} v(y)  |z| 1_{\{ |z| > 1 \}} J_{\xi}(dy,dz) < \infty. \\
\end{split}
\]
Consider the pure‐jump part of \(\dot\xi\) defined by 
\[Z(A)=\int_{D\times\{|z|\le1\}}1_A(y)\,z\,\widetilde J_{\xi}(dy,dz)
+\int_{D\times\{|z|>1\}}1_A(y)\,z\,J_{\xi}(dy,dz).\]
Then, 
\(u(x)=I_b(x)+I_W(x)+\langle\dot Z,\,G_{\cL}(x,\cdot)\rangle,\)
with 
\(\langle\dot Z,\,G_{\cL}(x,\cdot)\rangle=I_s(x)+I_l(x).\)
We apply Fubini’s theorem to \(I_b\), the standard Fubini theorem for \(L^2\)-random measures to \(I_W\) when $d \le 3$, and for the  term \(\langle\dot Z,\,G_{\cL}(x,\cdot)\rangle\), we apply the stochastic Fubini theorem for symmetric L\'evy noises given by Theorem 5.1 in \cite{DH}. Hence, the mild solution \(u\) coincides with the generalized solution \(u_{\mathrm{gen}}\).

\end{proof}

\subsection{Examples}
We will present the main elliptic operators to which Theorem \ref{existence-1} applies. Additionally, we extend the results of the previous section to a class of spectral operators.

The study of Green’s functions associated with elliptic operators dates back to early investigations of boundary value problems in the late eighteenth and nineteenth centuries, where they first appeared as integral kernels for the Laplace operator. Since then, Green’s functions have become indispensable in the analysis of both second-order and higher-order elliptic PDEs, and remain an active area of research up to the present time. A fundamental question in the theory of Green’s functions for elliptic operators, which continues to be actively studied, is to identify structural conditions on \(\mathcal{L}\) under which there exists a constant \(C>0\) such that
\[
G_{\mathcal{L}}(x,y)\;\le\;C\,G_{\Delta}(x,y),
\quad
x,y\in D,\;x\neq y.
\]

One of the earliest works to address this problem is \cite{littman}, later extended to bounded domains in \cite{GW1982}, which also provides our first example.

\begin{example}[\textit{Divergence operators}]
For $d \ge 3$, let $\cL$ be a divergence elliptic operator of the form
\begin{equation}
\label{div0}
\cL \equiv \textup{div}( K(x) \nabla ),
\end{equation}
where $K = ( K_{ij}(x) )$  is a symmetric, positive‐definite coefficient matrix $d \times d$. Hence, by Theorem 1.1 in \cite{GW1982}, \(G_{\cL}\) satisfies \eqref{d3g}, which implies that there exists a unique mild solution to \eqref{SPDE1} for a symmetric Lévy noise $\dot{\xi}$ with triplet \((b,0,\nu)\), and that \eqref{con-int} holds. For \(d=3\), we may take \(\dot{\xi}\) with triplet \((b,\sigma,\nu)\) without imposing any assumptions on the Lévy measure \(\nu\).

\end{example}

When \(\cL\) does not have the divergence form \eqref{div0}, it is well known that the Green function \(G_{\cL}\) does not satisfy the upper bound estimate \eqref{d3g}, even if \(A(x)\) is uniformly continuous, bounded, and elliptic (see, for instance, \cite{bauman84,GS1954}). Establishing when an estimate like \eqref{d3g} holds in the non-divergence case is a nontrivial problem. Recently, in \cite{HWKI} it was shown that if \(A(x)\) satisfies the following conditions, then \(G_{\cL}\) satisfies \eqref{d3g}. Define
\[
\bar{\textbf{A}}_{\tilde{D}_r(x)} \;:=\; \frac{1}{|{\tilde{D}_r} (x)|} \int_{{\tilde{D}_r} (x)} \textbf{A}(\xi)\,d\xi,
\quad
\chi_{\textbf{A}}(x,r) \;:=\; \int_{\tilde{D}_r (x)} \bigl\|\textbf{A}(y) - \bar{\textbf{A}}_{\tilde{D}_r (x)} (y) \bigr\|\,dy,
\]
where $\tilde{D}_r (x) = D \cap B_r (x)$. Additionally, set \(\chi_\textbf{A}(r) := \sup_{x \in \overline{D}} \chi_\textbf{A}(x,r)\). We say that \(\textbf{A}\) satisfies the \emph{Dini mean oscillation} condition on \(D\) if
\[
\int_0^1 \frac{\chi_\textbf{A}(t)}{t}\,dt < \infty.
\]

\begin{example}[\textit{Non-divergence operators}]
\label{non-div-ex}
For $d \ge 3$, let $\cL$ be a non-divergence operator of the form
\begin{equation}
    \label{non-div}
    \cL = \sum_{i, j =1}^d a_{i,j}(x) \partial_{x_i x_j} ,
\end{equation}
where $\textbf{A}(x) = ( a_{ij} (x) )$ satisfies the Dini mean oscillation condition on $D$. Hence, by Theorem 1.7 in \cite{HWKI}, we have that $G_{\cL}$ satisfies \eqref{d3g}. Therefore,there exists a unique mild solution to \eqref{SPDE1} for a symmetric Lévy noise with triplet \((b,0,\nu)\), and that \eqref{con-int} holds. For \(d=3\), we may take \(\dot{\xi}\) with triplet \((b,\sigma,\nu)\) without imposing any assumptions on the Lévy measure \(\nu\).  
\end{example}

Although the following operator is not elliptic in the strict sense, the existence of a mild solution to \eqref{SPDE1} follows from Theorem~\ref{existence-1}, which justifies presenting this problem in this section.

\title{\textit{Spectral fractional Laplace operator:}} We define the spectral power of \((-\Delta)\) of order \(\gamma>0\) by
\[
(-\Delta)^\gamma h \;=\; \sum_{k=1}^\infty \lambda_k^\gamma\,\cF_k[h]\,e_k.
\]
It can be shown that 
\[
(-\Delta)^\gamma: \bigcup_{r\in\mathbb{R}} H_r(D) \;\to\; \bigcup_{r\in\mathbb{R}} H_r(D).
\]
When \(\gamma=1\), this operator coincides with \(-\Delta\).  In what follows, we study \eqref{SPDE1} with \(\cL\) replaced by \((-\Delta)^\gamma\), i.e.
\begin{equation}
\label{SPDE-s}
\begin{cases}
(-\Delta)^\gamma u = \dot{\xi} \,& \text{in } D,\\
u = 0 \,& \text{on } \partial D.
\end{cases}
\end{equation}
Recall that we are interested in studying a mild solution \(u\) of \eqref{SPDE-s}, i.e.
\[
u(x) \;=\; \int_D G_{\gamma}(x,y)\,\xi(dy),
\]
where \(G_{\gamma}\) is the Green’s function of \((-\Delta)^\gamma\), given by the spectral representation:
\[
G_{\gamma}(x,y) \;=\; \sum_{k=1}^\infty \frac{e_k(x)\,e_k(y)}{\lambda_k^\gamma}.
\]
\begin{theorem}
\label{existence-s}
For \(\gamma > \frac{d}{4}\), equation \eqref{SPDE-s} admits a unique mild solution for any symmetric Lévy noise \(\dot{\xi}\) with characteristic triplet \((b,\sigma,\nu)\). Moreover, \(u\) is also a generalized solution to \eqref{SPDE1} in the sense that it satisfies \eqref{weak-s}.
\end{theorem}

\begin{proof}
We follow the same arguments as in the proof of Theorem~\ref{existence-1}, in particular using condition \eqref{cond-i-G} with \(p=2\).  It therefore suffices to show that
\begin{equation}
\label{s-2}
\int_D G_\gamma^2(x,y) \,dy < \infty.
\end{equation}
By orthogonality,
\[
\int_D G_\gamma(x,y)^2\,dy
= \sum_{k=1}^\infty \frac{|e_k(x)|^2}{\lambda_k^{2\gamma}}
= \int_{\lambda_1}^\infty t^{-2\gamma}\,V(dt,x),
\]
where
\[
V(t,x) := \sum_{\lambda_k \le t} |e_k(x)|^2.
\]
Since \(\partial D\) is smooth, Theorem 8.2 of \cite{agmon65} gives
\[
V(t,x)\;\le\;C\,t^{\frac d2},
\quad\forall\,x\in D.
\]
Hence, by Abel’s summation formula (see Theorem 4.2 in \cite{apostol76}) and \(\gamma>\frac{d}{4 }\),
\[
\int_{\lambda_1}^\infty t^{-2\gamma}\,V(dt,x)
\;\le\;C\int_{\lambda_1}^\infty u^{-2\gamma-1+\frac d2}\,du
\;<\;\infty.
\]
\end{proof}

\section{Path-wise regularity}

In this section, we prove the pathwise regularity properties of the solution given in Theorem~\ref{existence-1}. Since the Gaussian case has been widely studied, we will omit the continuous and drift components of \(\dot{\xi}\), i.e.\ set \(b = \sigma = 0\).

\begin{theorem}
\label{path}
Under the assumptions of Theorem \ref{existence-s}, the solution \(u\) of \eqref{SPDE-s} satisfies
\[
u(\omega,\cdot)\in H_r(D)\quad\text{for almost all }\omega\in\Omega
\]
for every \(r<2\gamma-\tfrac{d}{2}\). Moreover, if $\gamma \in (\frac{d}{2}, \infty)$, then $u$ admits a continuous modification in $\overline{D}$.
\end{theorem}

\begin{proof}
For \(\gamma > \frac{d}{4}\),  by Theorem 5.1 in \cite{DH}, for each $k \in \bN$,
\[
\cF_k [ \langle \dot{\xi}, G_\gamma (x, \cdot) \rangle] = \langle \dot{\xi}, \frac{e_k}{\lambda_k^\gamma} \rangle \quad \mbox{a.s.}
\]
Consequently, for each \(k\in\mathbb{N}\) there exists an event \(\Omega_k\subset\Omega\) with \(\mathbb{P}(\Omega_k)=1\) such that, for all \(\omega\in\Omega_k\),
\[
\mathcal{F}_k[u(\omega,\cdot)]
= \int_D \frac{e_k(y)}{\lambda_k^\gamma}\,\xi(\omega,dy)
= \int_{D\times\{|z|\le1\}}\frac{e_k(y)}{\lambda_k^\gamma}\,z\,\widetilde J_{\xi}(\omega,dy,dz)
+ \int_{D\times\{|z|>1\}}\frac{e_k(y)}{\lambda_k^\gamma}\,z\,J_{\xi}(\omega,dy,dz).
\]
Setting \(\displaystyle \Omega_\infty = \bigcap_{k=1}^\infty \Omega_k\), it follows that \(\mathcal{F}_k[u]\) is well-defined for all \(k\in\mathbb{N}\) on \(\Omega_\infty\). Hence, by Hölder’s inequality,
\[
\begin{split}
\| u \|_{H_r(D)}^2 
&\leq 2 \Biggl(
   \sum_{k\ge1} \lambda_k^r \Bigl|\!\int_{D\times\{|z|\le1\}}\frac{e_k(y)}{\lambda_k^\gamma}\,z\,\widetilde J_{\xi}(dy,dz)\Bigr|^2
   \;+\;
   \sum_{k\ge1} \lambda_k^r \Bigl|\!\int_{D\times\{|z|>1\}}\frac{e_k(y)}{\lambda_k^\gamma}\,z\,J_{\xi}(dy,dz)\Bigr|^2
\Biggr)\\
&\leq 2 \Biggl(
   \sum_{k\ge1} \lambda_k^{\,r-2\gamma} \Bigl|\!\int_{D\times\{|z|\le1\}} e_k(y)\,z\,\widetilde J_{\xi}(dy,dz)\Bigr|^2\\
&\qquad\quad
   +\;J_{\xi}(D\times\{|z|>1\})
        \int_{D\times\{|z|>1\}} 
          \sum_{k\ge1} \lambda_k^{\,r-2\gamma}\,\lvert e_k(y)\rvert^2\,\lvert z\rvert^2
        \;J_{\xi}(dy,dz)
\Biggr).
\end{split}
\]
Observe that
\[
\int_{D\times\{|z|>1\}} 
    \sum_{k=1}^\infty \lambda_k^{\,r-2\gamma}\,\lvert e_k(y)\rvert^2\,\lvert z\rvert^2
    \,J_{\xi}(dy,dz)
< \infty.
\]
By (A.3)-Lemma A.1 in \cite{CD23},
\[
\sup_{y\in D}\sum_{k=1}^\infty \lambda_k^{\,r-2\gamma}\,\lvert e_k(y)\rvert^2
< \infty
\quad\text{for }r<2\gamma-\tfrac d2.
\]

Similarly, for \(r<2\gamma-\tfrac d2\),
\[
\sum_{k=1}^\infty \lambda_k^{\,r-2\gamma}
\Bigl|\!\int_{D\times\{|z|\le1\}} e_k(y)\,z\,\widetilde J_{\xi}(dy,dz)\Bigr|^2
< \infty
\quad\text{a.s.}
\]
Indeed, by the It\^o isometry gives
\[
\mathbb{E}\Bigl[\Bigl|\!\int_{ D\times \{ |z|\le1 \}  } e_k(y)\,z\,\widetilde J_{\xi}(dy,dz)\Bigr|^2\Bigr]
= \|e_k\|_{L^2(D)}^2 \int_{\{ |z|\le1 \}}|z|^2\,\nu(dz),
\]
so that
\[
\sum_{k=1}^\infty \lambda_k^{\,r-2\gamma}\,
\mathbb{E}\Bigl[\Bigl|\!\int_{D\times\{|z|\le1\}} e_k(y)\,z\,\widetilde J_{\xi}(dy,dz)\Bigr|^2\Bigr]
= \Bigl(\int_{\{ |z|\le1 \}}|z|^2\,\nu(dz)\Bigr)
\sum_{k=1}^\infty \lambda_k^{\,r-2\gamma}
< \infty.
\]

Finally, we proved that \(u(\omega,\cdot)\in H_r(D)\) almost surely for every \(r<2\gamma-\tfrac{d}{2}\), by Theorem \ref{cont-embdd}, \(u(\omega,\cdot)\) is continuous on \(\overline{D}\) almost surely by choosing 
\[
\tfrac{d}{2}<r<2\gamma-\tfrac{d}{2}.
\]
Such an \(r\) exists precisely when \(\gamma > \tfrac{d}{2}\). Hence, the condition \(\gamma > \tfrac{d}{2}\) is necessary and sufficient for \(u\) to be continuous on \(\overline{D}\).

\end{proof}
The proofs of the following results adapt the argument of Theorem \ref{path}, using the same techniques in the eigenbasis of \(\cL\).

\begin{corollary}
Let \(\cL\) be an elliptic operator satisfying \eqref{d2g} or \eqref{d3g}. Then the solution \(u\) of \eqref{SPDE1} admits a version taking values in \(H_r(D)\) for every  $r \;<\; 2 - \frac{d}{2}.$
\end{corollary}

\begin{proposition}
Let \(u\) be the solution of \eqref{SPDE1} with \(\cL=\Delta\). Then \(u\) admits a continuous, bounded modification on $\overline{D}$ if and only if \(d=1\). If \(d\ge2\), then \(\|u\|_{\infty}=\infty\) a.s.
\end{proposition}

\begin{proof}
The existence of a continuous modification of \(u\) when \(\gamma=1\) follows immediately from Theorem \ref{path}. Alternatively, one can argue directly as follows. Since the Green function \(G_\Delta\) is continuous on \(D\times D\) only for \(d=1\), by the L\'evy It\^o decomposition, we have:
\[
u(x)
= \int_{D\times\{|z|\le1\}} G_\Delta(x,y)\,z\,\widetilde J_\xi(dy,dz)
+ \int_{D\times\{|z|>1\}} G_\Delta(x,y)\,z\,J_\xi(dy,dz).
\]
The second term is a finite sum of continuous functions of \(x\) when \(d=1\), and hence is continuous. For the first term, Kolmogorov’s continuity criterion guarantees a Hölder continuous modification on \(D\).

On the other hand, if \(d\ge2\), then \(G_\Delta\) has singularities along the diagonal \(x=y\). In this case, the “large‐jumps” term
\[
x\;\longmapsto\;\int_{D\times\{|z|>1\}} G_\Delta(x,y)\,z\,J_\xi(dy,dz)
\]
has exactly $J_\xi\bigl(D\times\{|z|>1\}\bigr)$ singularities, located at the atoms of the random measure \(J_\xi\).

\end{proof}

{\bf Acknowledgement.} The author is grateful to Yimin Xiao for drawing their attention to references \cite{CAD22, MR05}.

\end{document}